\documentclass[a4paper,12pt,onecolumn]{article}
\usepackage{etex}

\usepackage{hyperref}
\usepackage[vmargin=2cm,hmargin=2cm,headheight=14.5pt,top=2cm,headsep=.5cm]{geometry}

\usepackage{bm}
\usepackage{empheq}
\usepackage{stackrel}
\usepackage{cases}
\usepackage{mathtools}
\usepackage{amsthm,amsmath,amscd}
\usepackage{amssymb,amsfonts}
\usepackage{makeidx}
\usepackage[all]{xy}
\usepackage{perpage}
\usepackage{tikz-cd}
\tikzcdset{every label/.append style = {font = \small}}
\usepackage[symbol]{footmisc}

\MakePerPage[2]{footnote}
\usepackage{graphicx}
\DeclareMathSizes{12}{12}{8}{6}

\usepackage{cite}
\usepackage{url}
\usepackage[charter]{mathdesign}
\usepackage{accents}

\newtheoremstyle{ptheorem}{1em}{0em}{\itshape}{}{\bfseries}{.}{.5em}{\thmname{#1}\thmnumber{ #2}\thmnote{ (\hspace{-.01pt}{#3})}}

\theoremstyle{ptheorem}

\newtheorem{thm}{Theorem}[section]

\newtheorem{lem}[thm]{Lemma}

\newtheoremstyle{hdef}{1em}{0em}{}{}{\bfseries}{.}{.5em}{\thmname{#1}\thmnumber{ #2}\thmnote{ (\hspace{-.01pt}{#3})}}
\theoremstyle{hdef}

\newtheorem{rem}[thm]{Remark}

\makeatletter
\newtheoremstyle{premark}{1em}{0em}{
\addtolength{\@totalleftmargin}{1.5em}
\addtolength{\linewidth}{-1.5em}
\parshape 1 1.5em \linewidth}{}{\scshape}{.}{.5em}{}
\makeatother

\theoremstyle{premark}

\newtheorem{exa}[thm]{Example}

\numberwithin{equation}{section}
\numberwithin{figure}{section}

\DeclareMathOperator{\sign}{sign}

\DeclareMathOperator{\Id}{Id}

\DeclareMathOperator{\dif}{d}

\newcommand{\cL}{{\mathcal L}}
\newcommand{\cM}{{\mathcal M}}

\newcommand{\bN}{{\mathbb N}}

\newcommand{\bR}{{\mathbb R}}

\renewcommand{\a}{\alpha}
\renewcommand{\b}{\beta}
\renewcommand{\c}{\gamma}
\renewcommand{\l}{\lambda}

\renewcommand{\phi}{\varphi}

\newcommand{\ol}{\overline}

\renewcommand{\d}{\delta}

\renewcommand{\(}{\left(}
\renewcommand{\)}{\right)}

\newcommand{\til}{\tilde}
\newcommand{\Lsp}[1]{\operatorname{L^{#1}}}
\newcommand{\olb}[1]{%
  \vbox{\offinterlineskip\ialign{\hfil##\hfil\cr $\rotatebox[origin=c]{90}{$]$}$\cr\noalign{\kern-.45ex}{$#1$}\cr}}}
\begin{document}
\title{Computation of Green's functions through\\algebraic decomposition of operators\footnote{Partially supported by Conseller\'ia de Cultura, Educaci\'on e Ordenaci\'on Universitaria, Xunta de Galicia, Spain, project EM2014/032.}}

\author{
F. Adri\'an F. Tojo\footnote{Supported by  FPU scholarship, Ministerio de Educaci\'on, Cultura y Deporte, Spain.} \\
\normalsize
Departamento de An\'alise Ma\-te\-m\'a\-ti\-ca, Facultade de Matem\'aticas,\\ 
\normalsize Universidade de Santiago de Com\-pos\-te\-la, Spain.\\ 
\normalsize e-mail: fernandoadrian.fernandez@usc.es}
\date{}

\maketitle

\begin{abstract}
In this article we use linear algebra to improve the computational time for the obtaining of Green's functions of linear differential equations with reflection (DER). This is achieved by decomposing both the `reduced' equation (the ODE associated to a given DER) and the corresponding two-point boundary conditions.
\end{abstract}

\renewcommand{\abstractname}{Acknowledgements}
\begin{abstract}
	The author wants to acknowledge his gratitude towards the anonymous referee for helping improve the manuscript, specially in the proof of Lemma \ref{lemqq-}. 
\end{abstract}

\noindent {\small{\textbf{Classification: 34B, 47L, 34K.} \\ 
		\\
		\noindent {\small{\textbf{Competing Interests:} The author declares to not have any competing interests.} \\ 

\section{Introduction}
Differential operators with reflection have recently been of great interest, partly due to their applications to Supersymmetric Quantum Mechanics \cite{Post, Roy, Gam} or topological methods applied to nonlinear analysis \cite{Cab5}.\par
In the last years the works in this field have been related to either the obtaining of eigenvalues and explicit solutions of different problems \cite{PiaoSun,Pia3,Krits,Krits2}, their qualitative properties \cite{Ashy,Cab5} or the obtaining of the associated Green's function \cite{Sars,Sars2,Cab4,CabToj2,CabToj,Toj3,CTMal}. In \cite{CTMal} the authors described a method to derive the Green's function of differential equations with constant coefficients, reflection and two-point boundary conditions. This algorithm was implemented in \textit{Mathematica} (see \cite{Math}) in order to put it to a practical use. Unfortunately, it was soon observed that, although theoretically correct, there were severe limitations when it came to compute the Green's functions of problems of high order. In this respect, we have to point out that an order $n$ linear DER is reduced to an order $2n$ ordinary differential equation --see Theorem \ref{thmdei} and compare equations \eqref{rbvp} and \eqref{redpro}. This particularity posses a great challenge, for the computational time increases greatly with $n$.\par
To sort this out, the best option is to go back from an order $2n$ problem to \emph{two} problems of order $n$. This procedure, compared to solving directly the order $2n$, is much faster. Furthermore, it also happens that, in some cases, the decomposition provides two equivalent problems or a problem and its adjoint. In those cases the improvement is even more notorious.
\par
In the next Section we contextualize the problem with a brief introduction to differential equations with reflection and state some basic results concerning the Green's function associated to them.  In Section 3 we develop some theoretical results which provide a way of decomposing the DER we are dealing with. Finally, in Section 4 we establish a suitable decomposition for the boundary conditions, state criteria for self-adjointness of the decomposed problem and provide examples to illustrate the theory.

\section{Differential equations with reflection}
In order to establish a useful framework to work with these equations, we consider the differential operator $D$, the pullback operator of the reflection $\phi(t)=-t$, denoted by $\phi^*(u)(t)=u(-t)$, and the identity operator, $\Id$. \par
Let $T\in\bR^+$ and $I:=[-T,T]$. We now consider the algebra $\bR[D,\phi^* ]$ consisting of the linear operators of the form
\begin{equation}\label{Lop}L=\sum_{k=0}^n\(a_k\phi^*+b_k\)D^k.\end{equation}
where $n\in\bN$, $a_k,b_k\in\bR,\ k=1,\dots,n$, which act as
\begin{displaymath}Lu(t)=\sum_{k=0}^na_ku^{(k)}(-t)+\sum_{k=0}^nb_ku^{(k)}(t),\ t\in I,\end{displaymath}
on any function $u\in W^{n,1}(I)$.
The operation in the algebra is the usual composition of operators; we will omit the composition sign. We observe that $D^k\phi^*=(-1)^k\phi^*D^k$ for $k=0,1,\dots$, which makes it a \textit{noncommutative algebra}. 
We will consider, for convenience, the sums  $\sum_{k=0}^n\equiv\sum_{k}$ such that $k\in\{0,1,\dots\}$, but taking into account that the coefficients $a_k,b_k$ are zero for big enough indices.\par
Notice that $\bR[D,\phi^*]$ is not a unique factorization domain. For instance, \begin{displaymath}D^2-1=(D+1)(D-1)=-(\phi^*D+\phi^*)^2.\end{displaymath} \par
Let $\bR[D]$ be the ring of polynomials with real coefficients on the variable $D$. The following property is crucial for the obtaining of a Green's function.
\begin{thm}[{\cite[Theorem 2.1]{CTMal}}]\label{thmdec}
Take $L$ as defined in \eqref{Lop} and define
\begin{equation}\label{Rop}R:=\sum_{k}a_k\phi^*D^k+\sum_{l}(-1)^{l+1}b_lD^l\in\bR[D,\phi^* ].\end{equation}  Then $RL=LR\in\bR[D]$.
\end{thm}
\begin{rem}\label{remcoefred}
If $S:=RL=\sum_{k=0}^{2n} c_kD^k$, then
\begin{displaymath}c_k=\begin{dcases} 0, & k \text{ odd,} \\
2\sum_{l=0}^{\frac{k}{2}-1}\(-1\)^l\(a_la_{k-l}-b_lb_{k-l}\)+\(-1\)^\frac{k}{2}\(a_\frac{k}{2}^2-b_\frac{k}{2}^2\) & k \text{ even.}\end{dcases}\end{displaymath}
\end{rem}
This implies that the \textit{reduced} operator $RL$ has only coefficients for the even powers of the derivative, so the equation is self-adjoint. If the boundary conditions are appropriate (we will clarify this statement in Theorem \ref{libcab}), then the Green's function is symmetric \cite{Cablibro}. Observe that $c_0=a_0^2-b_0^2$. Also, if $L=\sum_{i=0}^n \(a_i\phi^*+b_i\)D^i$ with $a_n\ne0$ or $b_n\ne0$, we have that $c_{2n}=(-1)^n(a_n^2-b_n^2)$. Hence, if $a_n=\pm b_n$, then $c_{2n}=0$. This shows that composing two elements of $\bR[D,\phi^* ]$ we can get another element which has simpler terms in the sense of derivatives of less order. This is quite a difficulty when it comes to compute the Green's functions, for in this case we could have one, many or no solutions of our original problem \cite{CTMal}. The following example is quite illustrative.\par
\begin{exa}\label{firstexa} Consider the equation
\begin{displaymath}x^{3)}(t)+x^{3)}(-t)=\sin t,\ t\in I.\end{displaymath}
This equation cannot have a solution, for the left hand side is an even function while the right hand side is an odd function.\end{exa}
\par
As we said before, $S=RL$ is a usual differential operator with constant coefficients. Consider now the following problem.
\begin{equation}\label{lccbvp}Su(t):=\sum_{k=0}^na_ku^{k)}(t)=h(t),\ t\in I,\ B_ku:=\sum_{j=0}^{n-1}\left[\a_{kj}u^{j)}(-T)+\b_{kj}u^{j)}(T)\right]=0,\ k=1,\dots,n.
\end{equation}
The existence of Green's fuctions for problems such as \eqref{lccbvp} is a classical result (see, for instance, \cite{Cab6}). We present it here adapted to our framework.
\begin{thm}\label{thmgf}
Assume the following homogeneous problem has a unique solution
\begin{equation*}Su(t)=0,\ t\in I,\ B_ku=0,\ k=1,\dots n.
\end{equation*}
Then there exists a unique function, called \textbf{Green's function}, such that
\begin{itemize}
\item[(G1)] $G$ is defined on the square $I^2$.
\item[(G2)] The partial derivatives $\frac{\partial^kG}{\partial t^k}$ exist and are continuous on $I^2$ for $k=0,\dots,n-2$.
\item[(G3)] $\frac{\partial^{n-1}G}{\partial t^{n-1}}$ and $\frac{\partial^nG}{\partial t^n}$ exist and are continuous on $I^2\backslash\{(t,t)\ :\ t\in I\}$.
\item[(G4)] The lateral limits $\frac{\partial^{n-1}G}{\partial t^{n-1}}(t,t^+)$ and $\frac{\partial^{n-1}G}{\partial t^{n-1}}(t,t^-)$ exist for every $t\in(a,b)$ and
\begin{displaymath}\frac{\partial^{n-1}G}{\partial t^{n-1}}(t,t^-)-\frac{\partial^{n-1}G}{\partial t^{n-1}}(t,t^+)=\frac{1}{a_n}.\end{displaymath}
\item[(G5)] For each $s\in(a,b)$ the function $G(\cdot,s)$ is a solution of the differential equation $Su=0$ on $I\backslash\{s\}$.
\item[(G6)] For each $s\in(a,b)$ the function $G(\cdot,s)$ satisfies the boundary conditions $B_ku=0,\ k=1,\dots,n$.
\end{itemize}
Furthemore, the function $u(t):=\int_a^bG(t,s)h(s)\dif s$ is the unique solution of problem \eqref{lccbvp}.
\end{thm}
Now we can state the result which relates the Green's function of a problem with reflection to the Green's function of its associated reduced problem.\par
In order to do that, given an operator $\cL$ defined on some set of functions of one variable, we will define the operator $\cL_\vdash$ as $\cL_\vdash G(t,s):=\cL(G(\cdot,s))|_{t}$ for every $s$ and any suitable function $G$ of two variables.
\begin{thm}[\cite{CTMal}]\label{thmdei} Let $I=[-T,T]$. Consider the problem
\begin{equation}\label{rbvp}Lu(t)=h(t),\ t\in I,\ B_iu=0,\ k=1,\dots,n,
\end{equation}
where $L$ is defined as in \eqref{Lop}, $h\in L^1(I)$ and
\begin{displaymath}B_ku:=\sum_{j=0}^{n-1}\left[\a_{kj}u^{j)}(-T)+\b_{kj}u^{j)}(T)\right],\ k=1,\dots,n.\end{displaymath}
Then, there exists $R\in \bR[D,\phi^* ]$ (as in \eqref{Rop}) such that $S:=RL\in\bR[D]$ and the unique solution of problem \eqref{rbvp} is given by $\int_a^bR_\vdash G(t,s)h(s)\dif s$ where $G$ is the Green's function associated to the problem 
\begin{align}\label{redpro}Su & =0,\\ \label{redproc1}B_ku & =0,\ k=1,\dots,n, \\ \label{redproc2} B_kRu & =0,\ k=1,\dots,n,\end{align}
 assuming it has a unique solution.
\end{thm}
As stated in Section 1, Theorem \ref{thmdei} was implemented in Mathematica in \cite{Math}. We now proceed to describe some steps which could be added to the algorithm in order to improve it.
\section{Decomposing the reduced equation}
The computation of Green's functions is prohibitive in computation time terms \cite{Math}, mostly for high order equations, so it is necessary to find ways to palliate this problem. Our approach will consist of decomposing our problem in order to deal with equations of less order.\par
First observe that, from Remark \ref{remcoefred}, we know that the reduced equation has no derivatives of odd indices. For convenience, if $p$ is a real (complex) polynomial, we will denote by $p_-$ the polynomial with the same principal coefficient and opposite roots.
\begin{lem}\label{lemqq-}
Let $n\in\bN$ and $p(x)=\sum_{k=0}^{n}\a_{2k}x^{2k}$ a real polynomial of order $2n$. Then there is a complex polynomial $q$ of order $n$ such that $p=\a_{2n}qq_-$. Furthermore, if  $\til p(x)=\sum_{k=0}^{n}\a_{2k}x^{k}$ has no negative roots, $q$ is a real polynomial.
\end{lem}
\begin{proof}

	First observe that $p$ is a polynomial on $x^2$, and therefore, if $\l$ is an root of $p$, so has to be $-\l$. Hence, using the Fundamental Theorem of Algebra, the first part of the result can be derived by separating the monomials that compose $p$ in two different polynomials with opposite roots.\par Let us do that explicitly to show how in the case $\til p$ has no negative roots, $q$ is a real polynomial.\par
Take the change of variables $y=x^2$. Then, $p(x)=\til p(y)$ and, by the Fundamental Theorem of Algebra,
\begin{align*} \til  p(y)=  \sum_{k=0}^{n}\a_{2k}y^{k} = &  \a_{2n}y^\sigma(y-\l_1^2)\cdots(y-\l_m^2)(y+\l_{m+1}^2) \\ & \cdots (y+\l_{\ol m}^2)(y^2+\mu_1y+\nu_1^2)  \cdots(y^2+\mu_ly+\nu_l^2),\end{align*}
for some integers $\sigma,m,\ol m,l$ and real numbers $\l_1,\dots,\l_{\ol m},\nu_1,\dots,\nu_l,\mu_1,\dots,\mu_l$ such that $\l_k>0$ and $\nu_k>|\mu_k|/2$ for every $k$ in the appropriate set of indices. The terms of the form $y^2+\mu_ky+\nu_k^2$ correspond to the pairs of complex roots of the polynomial. This means that the discriminant $\Delta=\mu_k^2-4\nu_k<0$, that is, $\nu_k>|\mu_k|/2$.\par Hence,
\begin{align*}p(x)= & \a_{2n}x^{2\sigma}(x^2-\l_1^2)\cdots(x^2-\l_m^2)(x^2+\l_{m+1}^2) \\ & \cdots(x^2+\l_{\ol m}^2)(x^4+\mu_1x^2+\nu_1^2)\cdots(x^4+\mu_lx^2+\nu_l^2).\end{align*}
Now we have that
\begin{displaymath}(x^2-\l_k^2)=(x+\l_k)(x-\l_k),\quad (x^2+\l_k^2)=(x+\l_ki)(x-\l_ki),\end{displaymath}
\begin{displaymath}\text{and}\quad (x^4+\mu_kx^2+\nu_k^2)=(x^2-x \sqrt{2\nu_k-\mu_k}+\nu_k)(x^2+x \sqrt{2\nu_k-\mu_k}+\nu_k),\end{displaymath}
for any $k$ in the appropriate set of indices. Define
\begin{align*}q(x)= & x^\sigma(x-\l_1)\cdots(x-\l_m)(x-\l_{m+1}i)\cdots(x-\l_{\ol m}i)(x^2-x \sqrt{2\nu_1-\mu_1}+\nu_1) \\ & \cdots(x^2-x \sqrt{2\nu_l-\mu_l}+\nu_l),\end{align*}
and
\begin{align*}q_-(x)= & x^\sigma(x+\l_1)\cdots(x+\l_m)(x+\l_{m+1}i)\cdots(x+\l_{\ol m}i)(x^2+x \sqrt{2\nu_1-\mu_1}+\nu_1) \\ & \cdots(x^2+x \sqrt{2\nu_l-\mu_l}+\nu_l).\end{align*}
We have that $p=\a_{2n}qq_-$. 

Observe that if $\l$ is a root of $p$, $\l^2$ is a root of $\til p$. Hence, if $\til p$ has no negative roots, this is equivalent to $p$ not having roots of the form $\l=ai$ with $a\ne 0$. Thus,
\begin{align*}p(x)= & \a_{2n}x^{2\sigma}(x^2-\l_1^2)\cdots(x^2-\l_m^2)(x^4+\mu_1x^2+\nu_1^2)\cdots(x^4+\mu_lx^2+\nu_l^2),
\\
q(x)= & x^\sigma(x-\l_1)\cdots(x-\l_m)(x^2-x \sqrt{2\nu_1-\mu_1}+\nu_1)  \cdots(x^2-x \sqrt{2\nu_l-\mu_l}+\nu_l),
\\
q_-(x)= & x^\sigma(x+\l_1)\cdots(x+\l_m)(x^2+x \sqrt{2\nu_1-\mu_1}+\nu_1) \cdots(x^2+x \sqrt{2\nu_l-\mu_l}+\nu_l).
\end{align*}
That is, $q$ is a real polynomial.
\end{proof}
\begin{rem}\label{RemDes} Descartes' rule of signs establishes that the number of positive roots (with multiple roots counted separately) of a real polynomial on one variable is either equal to the number of sign differences between consecutive nonzero coefficients, or less than it by an even number, considering the case the terms of  the polynomial are ordered by descending variable exponent. This implies that a sufficient criterion for a polynomial $p(x)$ to have no negative roots is for $p(-x)$ to have all coefficients  with positive sign, that is, for $p(x)$ to have positive even coefficients and negative odd coefficients.\par
There exist algorithmic ways of determining the exact number of positive (or real) roots of a polynomial. For more information on this issue see, for instance, \cite{Yan1,Yan2, Lia}.
\end{rem}
 The following Lemma establishes a relation between the coefficients of $q$ and $q_-$.
\begin{lem}\label{relqq-}
Let $n\in\bN$ and $q(x)=\sum_{k=0}^{n}\a_{k}x^{k}$ be a complex polynomial. Then \[q_-(x)=\sum_{k=0}^{n}(-1)^{k+n}\a_{k}x^{k}.\]
\end{lem}
\begin{proof}
We proceed by induction. For $n=1$, $q(x)=\a(x-\l_1)$. Clearly, $q$ has the root $\l_1$ and $q_-(x)=\a(x+\l_1)=(-1)^{1+1}\a x+(-1)^{1}\a\l_1$ the root $-\l_1$.\par
Assume the result is true for some $n\ge1$. Then, for $n+1$, $q$ is of the form $q(x)=(x-\l_{n+1})r(x)$ where $r(x)=\sum_{k=0}^{n}\a_{k}x^{k}$ is a polynomial of order $n$, that is,
\begin{displaymath}q(x)=(x-\l_{n+1})\sum_{k=0}^{n}\a_{k}x^{k}=x^{n+1}+\sum_{k=1}^n\left[\a_{k-1}-\l_{n+1}\a_{k}\right]x^k-\l_{n+1}\a_0.\end{displaymath}
Now, $q_-(x)=(x+\l_{n+1})r_-(x)$. Since the formula is valid for $n$,
\begin{align*}q_-(x) & =(x+\l_{n+1})r_-(x)=(x+\l_{n+1})\sum_{k=0}^{n}(-1)^{k+n}\a_{k}x^{k}\\ & =x^{n+1}+\sum_{k=1}^n(-1)^{k+n+1}\left[\a_{k-1}-\l_{n+1}\a_{k}\right]x^k-(-1)^{n+1}\l_{n+1}\a_0.\end{align*}
So the formula is valid for $n+1$ as well.
\end{proof}
\begin{rem}The result can be directly proven by considering the last statement in Remark \ref{RemDes}. If we take a polynomial $p(x)=a(x-\l_1)\cdots(x-\l_n)$, the polynomial $p(-x)$ has exactly opposite roots. Actually, $p(-x)=a(-x-\l_1)\cdots(-x-\l_n)=(-1)^na(x+\l_1)\cdots(x+\l_n)$. It is easy to check that the coefficients of $p(-x)$ are precisely as described in the statement of Lemma \ref{relqq-} save for the factor $(-1)^n$.
\end{rem}
This last Lemma allows the computation of the polynomials $q$ and $q_-$ related to the polynomial $RL$ on the variable $D$ using the formula given in Remark \ref{remcoefred}. We will assume that $RL$ is of order $2n$, that is, $a_n^2-b_n^2\not=0$. Otherwise the problem of computing $q$ and $q_-$ would be the same but these polynomials would be of less order. Also, assume $RL$, considered as a polynomial on $D^2$, has no negative roots in order for $q$ to be a real polynomial. If $L=\sum_{k=0}^{n}(a_{k}\varphi^*+b_k)D^{k}$ and  $q(D)=D^n+\sum_{k=0}^{n-1}\a_{k}D^{k}$ then \[RL=\sum_{k=0}^{2n}c_kD^k=(-1)^n(a_n^2-b_n^2)q(D)q_-(D).\]
This relation establishes the following system of quadratic equations:
\begin{align*} c_{2k}= & 2\sum_{l=0}^{k-1}\(-1\)^l\(a_la_{2k-l}-b_lb_{2k-l}\)+\(-1\)^k\(a_k^2-b_k^2\) \\ = & (a_n^2-b_n^2)\left[2\sum_{l=0}^{k-1}\(-1\)^l\(\a_l\a_{2k-l}\)+\(-1\)^k\a_k^2\right],\quad k=0,\dots,n,\end{align*}
where $a_k,b_k,\a_k=0$ if $k\not\in\{0,\dots,n\}$ and $\a_n=1$. These are $n$ equations with $n$ unknowns: $\a_0,\dots,\a_n$. We present here the case of $n=2$ to illustrate the solution of these equations.
\begin{exa}
For $n=2$, we have that
\begin{align*}RL &= \left(a_2^2-b_2^2\right)D^4+ \left(-a_1^2+2 a_0 a_2+b_1^2-2 b_0 b_2\right)D^2+a_0^2-b_0^2,\\ \left(a_2^2-b_2^2\right)q(D)q_-(D) & =  \left(a_2^2-b_2^2\right)D^4+\left(2 \alpha _0-\alpha _1^2\right) \left(a_2^2-b_2^2\right)D^2+\alpha _0^2 \left(a_2^2-b_2^2\right),\end{align*}
and the system of equations is
\begin{equation} \label{eqsym2}\begin{aligned}
a_0^2-b_0^2 & =\left(a_2^2-b_2^2\right)\alpha _0^2,\\ -a_1^2+2 a_0 a_2+b_1^2-2 b_0 b_2 &= \left(a_2^2-b_2^2\right)\left(2 \alpha _0-\alpha _1^2\right).
\end{aligned}\end{equation}
Before computing the solutions let us state explicitly the limitations that the fact that $RL$, considered as an order 2 polynomial on $D^2$, that is, that $RL(x)=a x^2+b x +c$ has no negative roots implies. There are two options:
\begin{enumerate}
\item[\emph{(I)}] There are two complex roots, that is, $\Delta= b^2-4ac<0$. This is equivalent to $ac>0\land|b|<2\sqrt{ac}$. Expressed in terms of the coefficients of $RL$:
\begin{displaymath}(b_0^2-a_0^2) (b_2^2-a_2^2)>0 \text{ and } |-a_1^2+2 a_0 a_2+b_1^2-2 b_0 b_2|<2\sqrt{(b_0^2-a_0^2) (b_2^2-a_2^2)}.\end{displaymath}
\item[\emph{(II)}] There are two nonnegative roots, that is $\Delta=b^2-4ac\ge0$ and \begin{displaymath}(-b+\sqrt{b^2-4ac})/(2a)\le 0.\end{displaymath} This is equivalent to $(a,c\ge 0\land -b\ge2\sqrt{ac})\lor(a,c\le 0\land b\ge2\sqrt{ac})$. Expressed in terms of the coefficients of $RL$:
\end{enumerate}
\begin{align*}& \left[(b_0^2-a_0^2), (b_2^2-a_2^2)\ge0\land-(-a_1^2+2 a_0 a_2+b_1^2-2 b_0 b_2)\ge2\sqrt{(b_0^2-a_0^2) (b_2^2-a_2^2)}\right]\end{align*}\begin{center}{\textsc{or}}\end{center} \begin{align*} & \left[(b_0^2-a_0^2), (b_2^2-a_2^2)\le0\land-(-a_1^2+2 a_0 a_2+b_1^2-2 b_0 b_2)\ge 2\sqrt{(b_0^2-a_0^2) (b_2^2-a_2^2)}\right].\end{align*}

Now, with these conditions, the solutions the system of equations \eqref{eqsym2} are:

\emph{Case (I).} We have two solutions:
\begin{displaymath}\a_0=\sqrt{\frac{b_0^2-a_0^2}{b_2^2-a_2^2}},\end{displaymath} \begin{displaymath} \a_1=\pm\sqrt{\frac{2\sign(a_2^2-b_2^2)\sqrt{(b_0^2-a_0^2) (b_2^2-a_2^2)}-(-a_1^2+2 a_0 a_2+b_1^2-2 b_0 b_2)}{a_2^2-b_2^2}}.\end{displaymath}\par
\emph{Case (II).} We have four solutions depending on whether we choose $\xi=1$ or $\xi=-1$:
\begin{displaymath}\a_0=\xi\sqrt{\frac{b_0^2-a_0^2}{b_2^2-a_2^2}},\end{displaymath} \begin{displaymath} \a_1=\pm\sqrt{\frac{2\xi\sign(a_2^2-b_2^2)\sqrt{(b_0^2-a_0^2) (b_2^2-a_2^2)}-(-a_1^2+2 a_0 a_2+b_1^2-2 b_0 b_2)}{a_2^2-b_2^2}}.\end{displaymath}

These solutions provide well defined real numbers by conditions \emph{(I)} and \emph{(II)}.
\end{exa}
\section{Decomposing the boundary conditions}
Now we consider those cases where the problem can be decomposed in two equations. We will try to identify those circumstances when problem \eqref{redpro}-\eqref{redproc1}-\eqref{redproc2} can be expressed as an equivalent factored problem of the form
\begin{alignat}{2}\label{p1}L_1u & =y,\ && V_ju=0, j=1,\dots,n,\\\label{p2} L_2y & =Rh,\ && \widetilde{V_j}y=0, j=1,\dots,n,
\end{alignat}
where $S=L_2L_1$. If that where the case, we know the conditions \eqref{redproc1}-\eqref{redproc2} have to be equivalent to
\begin{equation}\label{vc}V_ju=0,\ \widetilde{V_j}L_1u=0,\ j=1,\dots,n.\end{equation} 
In this case, the Green's function of problem \eqref{redpro}-\eqref{redproc1}-\eqref{redproc2} can be expressed as 
\begin{displaymath}G(t,s)=\int_{-T}^TG_1(t,r)G_2(r,s)\dif r,\end{displaymath}
where $G_1$ is the Green's function associated to the problem \eqref{p1} and $G_2$ the one associated to the problem \eqref{p2} assuming both Green's functions exist.\par
In order to guarantee that \eqref{redproc1}-\eqref{redproc2} and \eqref{vc} are equivalent, let us establish the following definitions. Let \begin{align*}\Gamma_1: & =(\a_{kj})_{k  =1,\dots,n}^{j   =0,\dots,n-1},\quad X_n:=(u(T),u'(T),\dots,u^{(n-1)}(T)),\\\Theta_1: & =(\b_{kj})_{k  =1,\dots,n}^{j   =0,\dots,n-1},\quad\ol X_n:=(u(-T),u'(-T),\dots,u^{(n-1)}(-T)).\end{align*} Then the boundary conditions  \eqref{redproc1} can be expressed as $\Gamma_1\ol X_n+\Theta_1X_n=0$. In the same way, \eqref{redproc2} can be written as $(\Gamma_2\ \Gamma_3)\ol X_{2n}+(\Theta_2\ \Theta_3)X_{2n}=0$ for some matrices $\Gamma_2,\Gamma_3,\Theta_2,\Theta_3\in\cM_n(\bR)$. So, globally, the conditions on equation \eqref{redpro} can be expressed as
\begin{equation}\label{c1}\begin{pmatrix} \Gamma_1 & 0 \\ \Gamma_2 & \Gamma_3\end{pmatrix}\ol X_{2n}+\begin{pmatrix}\Theta_1 & 0 \\ \Theta_2 & \Theta_3\end{pmatrix} X_{2n}=0.\end{equation}
Now, assume $L_1$ and $\widetilde{V}_j$ are of the form
\begin{align*}L_1 & =\sum_{l=0}^{n}c_{l}D^{l},\\
 \widetilde{V}_ju & =\sum_{k=0}^{n-1}\left[\c_{jk}u^{k)}(-T)+\d_{jk}u^{k)}(T)\right]=\sum_{k=0}^{n-1}\left[\c_{jk}(-T)^*+\d_{jk}T^*\right]D^{k}u,\ j=1,\dots,n.\end{align*}
for some $c_l,\c_{jk},\d_{jk}\in\bR$, $l,j,k=1,\dots,n$ and where $a^*$ denotes the pullback by the constant $a$. Define now $\Phi:=(\gamma_{jk})_{j,k}$, $\Psi:=(\delta_{jk})_{j,k}\in\cM_n(\bR)$ and
\[\Xi=(d_{jk})_{j=0,\dots, n-1}^{k=0,\dots,2n-1}:=\begin{pmatrix}c_0 & c_1 & c_2 & \cdots & c_{n-1} & c_n & 0 & 0 & \cdots & 0 \\ 0& c_0 & c_1 & \cdots & c_{n-2} & c_{n-1} & c_n & 0 & \cdots & 0\\ 0 & 0& c_0 & \cdots & c_{n-3} & c_{n-2} & c_{n-1} & c_n & \cdots & 0 \\ \vdots & \vdots & \vdots & \ddots & \vdots & \vdots & \vdots & \vdots & \ddots & \vdots \\ 0 & 0 & 0 & \cdots & c_0 & c_1 & c_2 & c_3 & \cdots & c_n\end{pmatrix}=\(\Xi_1\ \Xi_2\)\in\cM_{n\times2n}(\bR),\]
where $\Xi_1$, $\Xi_2\in\cM_n(\bR)$, $\Xi_2$ is invertible (because $c_n\ne0$) and  $\Xi_1$ is invertible if and only if $c_0\ne0$.
\par Now we are ready to start the calculations. We have that

\begin{align*}(\widetilde{V}_jL_1u)_j  = & \(\sum_{k=0}^{n-1}\left[\c_{jk}(-T)^*+\d_{jk}T^*\right]D^{k}\sum_{l=0}^{n}c_{l}D^{l}u\)_j=\(\sum_{k=0}^{n-1}\sum_{l=0}^{n}\left[\c_{jk}c_{l}(-T)^*+\d_{jk}c_{l}T^*\right]D^{k+l}u\)_j \\ = &\(\sum_{k=0}^{n-1}\sum_{m=k}^{k+n}\left[\c_{jk}c_{m-k}(-T)^*+\d_{jk}c_{m-k}T^*\right]D^{m}u\)_j \\= & \(\sum_{k=0}^{n-1}\sum_{m=0}^{2n-1}\left[\c_{jk}d_{km}u^{(m)}(-T)+\d_{jk}d_{km}u^{(m)}(T)\right]\)_j\\= &\(\sum_{k=0}^{n-1}\c_{jk}d_{km}\)_{j,m}\ol X_{2n}+\(\sum_{k=0}^{m}\d_{jk}d_{km}\)_{j,m}X_{2n}=\Phi\Xi\ol X_{2n}+\Psi\Xi X_{2n}.
\end{align*}
Hence, we would write \eqref{vc} in the form
\begin{equation}\label{c2}\begin{pmatrix} \widetilde\Phi & 0 \\ \Phi\Xi_1 & \Phi\Xi_2\end{pmatrix}\ol X_{2n}+\begin{pmatrix}\widetilde\Psi & 0 \\ \Psi\Xi_1 & \Psi\Xi_2\end{pmatrix} X_{2n}=0.\end{equation}
Clearly, it is convenient to take $\widetilde\Phi=\Gamma_1$ and $\widetilde\Psi=\Theta_1$, that is, $V_j=B_j$, $j=1,\dots,n$.
\begin{lem}\label{invlem}
If $\Gamma_1$ and  $\Gamma_3$ are invertible and $\Theta_2=\Gamma_2\Gamma_1^{-1}\Theta_1+\Theta_3\Xi_2^{-1}\Xi_1-\Gamma_3\Xi_2^{-1}\Xi_1\Gamma_1^{-1}\Theta_1$, then, taking
\[\widetilde\Phi=\Gamma_1,\ \widetilde\Psi=\Theta_1,\ \Phi=\Id,\text{ and } \Psi=\Xi_2\Gamma_3^{-1}\Theta_3\Xi_2^{-1},\]
 condition \eqref{c1} is equivalent to condition \eqref{c2} and, therefore, problems \eqref{redpro}-\eqref{redproc1}-\eqref{redproc2} and \eqref{p1}-\eqref{p2} are equivalent.
\end{lem}
\begin{proof} Let
\[A=\begin{pmatrix}\Id & 0 \\ (\Xi_1-\Xi_2\Gamma_3^{-1}\Gamma_2)\Gamma_1^{-1} & \Xi_2\Gamma_3^{-1}\end{pmatrix}.\]
$A$ is invertible and 
\[\begin{pmatrix} \widetilde\Phi & 0 \\ \Phi\Xi_1 & \Phi\Xi_2\end{pmatrix}=A\begin{pmatrix} \Gamma_1 & 0 \\ \Gamma_2 & \Gamma_3\end{pmatrix}, \quad\begin{pmatrix}\widetilde\Psi & 0 \\ \Psi\Xi_1 & \Psi\Xi_2\end{pmatrix}=A\begin{pmatrix}\Theta_1 & 0 \\ \Theta_2 & \Theta_3\end{pmatrix}.\]

Hence, conditions \eqref{c1} and \eqref{c2} are equivalent.
\end{proof}
Analogously, we have a result where it is the $\Theta_1$ and $\Theta_3$ which are invertible.
\begin{lem}\label{invlem2}
If $\Theta_1$ and  $\Theta_3$ are invertible and $\Gamma_2=\Theta_2\Theta_1^{-1}\Gamma_1+\Gamma_3\Xi_2^{-1}\Xi_1-\Theta_3\Xi_2^{-1}\Xi_1\Theta_1^{-1}\Gamma_1$, then, taking
\[\widetilde\Psi=\Theta_1,\ \widetilde\Phi=\Gamma_1,\ \Psi=\Id,\text{ and } \Phi=\Xi_2\Theta_3^{-1}\Gamma_3\Xi_2^{-1},\]
 condition \eqref{c1} is equivalent to condition \eqref{c2} and, therefore, problems \eqref{redpro}-\eqref{redproc1}-\eqref{redproc2} and \eqref{p1}-\eqref{p2} are equivalent.
\end{lem}
 The following example illustrates this discussion explicitly.
\begin{exa}\label{exagf-2}
Consider the following problem.
\begin{equation}\begin{aligned}\label{prooc-2} & u'''(t)+u(-t)+u(t)=h(t),\ t\in I, \\ &   u(-1)-u''(1)=0,\ u'(-1)=u'(1),\  u''(-1)-u(1)=0,\end{aligned}
\end{equation}
 where $h(t)=\sin t$. Then, the operator we are studying is $L=D^3+\varphi^*+1$. If we take $R:=D^3+\varphi^*-1$, we have that $RL=D^6$, which admits a simple decompostion in $\bR[D]$ as $RL=(D^3)(D^3)=L_2L_1$.\par
The boundary conditions are
\[[(-1)^*-1^*D^2]u=0,\ [(-1)^*D-1^*D]u=0,\ [(-1)^*D^2-1^*]u=0.\] Taking this into account,  we add the conditions
\begin{align*}0 & =[(-1)^*-1^*D^2]Ru=u'''(-1)-u^{(5)}(1),\\
0 & =[(-1)^*D-1^*D]Ru=u^{(4)}(-1)-u^{(4)}(1),\\
0 & =[(-1)^*D^2-1^*]Ru=u^{(5)}(-1)-u'''(1).\\
\end{align*}
That is, our new \textit{reduced} problem, writing the boundary conditions in matrix form, is
\begin{equation}\label{proocr-2}\begin{aligned} & u^{(6)}(t)=f(t),\\    & \begin{pmatrix}
1 & 0 & 0 & 0 &0&0\\ 
0 & 1 & 0 & 0&0&0 \\ 
0 & 0 & 1 & 0&0&0 \\ 
0 & 0 & 0 & 1&0&0 \\ 
0 & 0 & 0 & 0&1&0 \\ 
0 & 0 & 0 & 0&0&1
\end{pmatrix}\begin{pmatrix}
u(-1) \\ u'(-1) \\ u''(-1) \\ u'''(-1)\\u^{(4)}(-1)\\u^{(5)}(-1)
\end{pmatrix}+\begin{pmatrix}
0&0 & -1 & 0 & 0&0 \\ 
0& -1 & 0 & 0 & 0 &0\\ 
-1 & 0 & 0 & 0&0&0 \\
0 & 0 & 0 & 0&0&-1 \\ 
0 & 0 & 0 & 0&-1&0 \\ 
0 & 0 & 0 & -1&0&0
\end{pmatrix}\begin{pmatrix}
u(1) \\ u'(1) \\ u''(1) \\ u'''(1)\\u^{(4)}(1)\\u^{(5)}(1)
\end{pmatrix}  =0 .\end{aligned}
\end{equation}
where $f(t)=R\,h(t)=h'''(t)+h(-t)-h(t)=-3\sin t$.\par
Now, we can check that we are working in the conditions of Lemma \ref{invlem}. We have that $\Gamma_1=\Gamma_3=\Id$, $\Gamma_2=\Theta_2=0$ and
\[\Theta_1=\Theta_3=\begin{pmatrix}
0&0&-1\\ 
0&-1&0 \\
-1&0&0
\end{pmatrix}. \]
On the other hand,
\[\Xi=\(\Xi_1\ \Xi_2\)=\begin{pmatrix}1 & 0 & 0 & 1 & 0 & 0 \\ 0& 1 & 0 & 0 & 1 & 0\\0&0& 1 & 0 & 0 & 1 \end{pmatrix}.\]
Thus, it is straightforward to check that
\[\Gamma_2\Gamma_1^{-1}\Theta_1+\Theta_3\Xi_2^{-1}\Xi_1-\Gamma_3\Xi_2^{-1}\Xi_1\Gamma_1^{-1}\Theta_1=\Theta_2=0,\]
and therefore the hypotheses of Lemma \ref{invlem} are satisfied. The conditions $\til V_j$ are given by the matrices $\Phi=\Id$ and $\Psi=\Xi_2\Gamma_3^{-1}\Theta_3\Xi_2^{-1}=\Theta_3$. Hence, we know that this problem is equivalent to factored system,
\begin{alignat}{3}
\label{exrp1}u'''(t) & =v(t),\ && u(-1)-u''(1)=0,\ u'(-1)=u'(1),\  u''(-1)-u(1)=0,\\
\label{exrp2}v'''(t) & =-3\sin t,\ &&  v(-1)-v''(1)=0,\ v'(-1)=v'(1),\  v''(-1)-v(1)=0.
\end{alignat}
Thus, it is clear that
\begin{displaymath}u(t)=\int_{-1}^1G_1(t,s)v(s)\dif s,\ v(t)=\int_{-1}^1G_2(t,s)f(s)\dif s,\end{displaymath}
where, $G_1=G_2$ are, respectively, the Green's functions of \eqref{exrp1} and \eqref{exrp2}. The Green's functions of problems involving linear ordinary differential equations with constant coefficients and two-point boundary conditions can be computed with the \emph{Mathematica} notebooks \cite{Math2} or \cite{Math}. Explicitly,
\begin{equation*}
G_1(t,s)=\begin{cases}  -\frac{1}{4} (s-t) (s (t-1)+t-3), & -1\leq s\leq t\leq 1, \\
 -\frac{1}{4} (s-t) ((s-1) t+s-3), & -1<t<s\leq 1. \\ \end{cases}
\end{equation*}

Hence, the Green's function $G$ for problem \eqref{proocr-2} is given by
\begin{align*}& G(t,s)  =\int_{-1}^1G_1(t,r)G_2(r,s)\dif r=\\ & \frac{1}{480} \begin{dcases}
 \begin{aligned} 2 s^5 (t+1)-5 s^4 (t (t+2)+3)+20 s^3 t (t+3)-5 s^2 \left(t^2 (t+2)^2-5\right) \\ +2 s t \left(t^2 (t (t+5)+30)-166\right)-2 t^5-15 t^4+25 t^2-102, \end{aligned} & -1<t<s\leq 1, \\
 \begin{aligned}-2 s^5-15 s^4-5 \left(s^2 (s+2)^2-5\right) t^2+2 \left(s^2 (s (s+5)+30)-166\right) s t \\ +25 s^2+2 (s+1) t^5-5 (s (s+2)+3) t^4+20 (s+3) s t^3-102, \end{aligned} & -1\leq s\leq t\leq 1.
\end{dcases}
\end{align*}
Therefore, using Theorem \ref{thmdei}, the Green's function for problem \eqref{prooc-2} is
\begin{align*} & \ol G(t,s) =R_\vdash G(t,s)  =\frac{\partial^3 G}{\partial t^3}(t,s)+G(-t,s)+G(t,s)=\\ &  \frac{1}{120}\begin{dcases}
 \begin{aligned} -(s-1) t^5+10 (s-3) s t^3+30 (s-1) t^2-30 (s-3) s \\ -\left(s^5-5 s^4+30 s^3+30 s^2-226 s+90\right) t, \end{aligned} & -1\le |t|\le s\le 1, \\
\begin{aligned} s^5 (-(t-1))+10 s^3 (t-3) t-30 s^2 (t-1)+30 (t-3) t \\ +s \left(-t^5+5 t^4-30 t^3+30 t^2+106 t+90\right),\end{aligned} & -1\le |s|< t\le 1, \\
\begin{aligned} s^5 (-(t+1))-10 s^3 t (t+3)-30 s^2 (t+1)-30 t (t+3) \\ -s \left(t^5+5 t^4+30 t^3-30 t^2-226 t-90\right),\end{aligned} & -1\le |s|< -t\le 1, \\
 \begin{aligned}-(s+1) t^5-10 s (s+3) t^3+30 (s+1) t^2+30 s (s+3) \\ -\left(s^5+5 s^4+30 s^3+30 s^2-106 s+90\right) t,\end{aligned} &-1\le |t|\le -s\le 1. \\
\end{dcases}\end{align*}
Hence, the solution of problem \eqref{prooc-2} is given by
\begin{align*}u(t)=  \int_{-1}^1\ol G(t,s)\sin(s)\dif s= & -\frac{1}{60} \left(-30 - 91 t - 30 t^2 + 10 t^3 + t^5\right) \sin (1) \\ & +\frac{2}{3} \left(t^3-7 t-3\right) \cos (1)+2 \sin (t)+\cos (t).\end{align*}
\end{exa}

Computationally, this procedure poses a big advantage: it is always easier to obtain the Green's function for two order $n$ problems than to do so for one order $2n$ problem. Furthermore, if the hypotheses of Lemma~\ref{lemqq-} are satisfied and we are able to obtain a factorization of the aforementioned kind using $q$ and $q_-$ in the place of $L_1$ and $L_2$, we have an extra advantage: the differential equation given by $q_-$ is the adjoint equation of the one given by $q$ multiplied by the factor $(-1)^n$. This fact, together with the following result --which can be found, although not stated as in this work, in \cite{Cablibro}, illustrates that in this case it may be possible to solve problem \eqref{rbvp} just computing the Green's function of one order $n$ problem.
\begin{thm}\label{libcab}
Consider an interval $J=[a,b]\subset\bR$, functions $\sigma,a_i\in\Lsp{1}(J)$, $i=1,\dots,n$, real numbers $\a_{ij},\b_{ij},h_i$, $i=1,\dots, n$, $j=0,\dots,n-1$,  $D(L_n)\subset W^{n,1}(J)$ a vector subspace, the operator \[L_nu(t)=a_0u^{(n)}(t)+a_1(t)u^{(n-1)}(t)+\dots+a_{n-1}(t)u'(t)+a_n(t)u(t),\ t\in J,\ u\in D(L_n),\]
with $a_0=1$ and the problem
\begin{equation}\label{pfeq} L_nu(t)=\sigma(t),\ t\in J,\quad U_i(u)=h_i,\ i=1,\dots,n,\end{equation}
where
\[U_i(u):=\sum_{j=0}^{n-1}\(\a_{ij}u^{(j)}(a)+\b_{ij}u^{(j)}(b)\),\quad i=1,\dots,n.\]
Then, the associated adjoint problem is
\begin{equation}\label{feq} L_n^\dagger v(t)=\sum_{j=0}^n(-1)^ja_{n-j}(t)u^{(j)}(t),\ t\in J,\ v\in D(L_n^\dagger),\end{equation}
where
\[D(L_n^\dagger)=\left\{v\in W^{n,2}(J)\ :\ (b^*-a^*)\(\sum_{j=1}^n\sum_{i=0}^{j-1}(-1)^{(j-i-1)}(a_{n-j}v)^{j-i-1}u^{(i)}\)=0,\ u\in D(L_n)\right\}.\]
Furthermore, if $G(t,s)$ is the Green's function of problem \eqref{pfeq}, then the one associated to problem~\eqref{feq} is $G(s,t)$.
\end{thm}

Hence, if we can decompose problem  \eqref{redpro}-\eqref{redproc1}-\eqref{redproc2} in two adjoint problems of the form \eqref{p1}-\eqref{p2}, its Green's function will be
\begin{displaymath}G(t,s)=\int_{-T}^TG_1(t,r)G_2(r,s)\dif r=\int_{-T}^TG_1(t,r)G_1(s,r)\dif r.\end{displaymath}
where $G_1$ is the Green's function of \eqref{p1} and $G_2(t,s)=G_1(s,t)$ the one of \eqref{p2}. We note though, that unless the operator $q_-$ is the adjoint equation times $(-1)^n$, the boundary conditions may not be the adjoint ones.
\begin{exa}Consider the problem
\begin{equation}\label{pex1}u'(-t)+u(t)+\sqrt{2}\,u(-t)=f(t):=e^t,\ t\in[-1,1],\ u(-1)=u(1),
\end{equation}
Taking $R=\phi^*D+\sqrt{2}\phi^*-\Id$ and composing problem \eqref{pex1} with this operator we obtain the reduced problem
\begin{equation}\label{pex2}u''(t)-u(t)=Rf(t),\ t\in[-1,1],\ u(-1)=u(1), \ u'(-1)=u'(1).
\end{equation}
Problem \eqref{pex2} is equivalent to the factored system
\begin{alignat}{3}
\label{ex2rp1}u'(t)+u(t) & =v(t), && u(-1) && =u(1),\\
\label{ex2rp2}-v'(t)+v(t) & =-Rf(t),\quad &&  v(-1) && =v(1).
\end{alignat}
for $t\in[-1,1]$.
Observe problem \eqref{ex2rp2} is the adjoint problem of  \eqref{ex2rp1}. Since the Green's function of problem  \eqref{ex2rp1} is given by
\[G_1(t,s):=\begin{dcases}
 \frac{e^{s-t+2}}{e^2-1}, & -1\leq s\leq t\leq 1, \\
 \frac{e^{s-t}}{e^2-1}, & -1<t<s\leq 1,\end{dcases}\]
 and, therefore, $G_1(s,t)$ is the Green's function of problem  \eqref{ex2rp2}, the Green's function of problem \eqref{pex2} is
 \[G(t,s)=-\int_{-1}^1G_1(t,r)G_1(s,r)\dif r=\begin{dcases} 
 - \frac{e^{s-t+2}+e^{t-s}}{2 e^2-2}, &  -1\leq s\leq t\leq 1, \\
  -\frac{e^{s-t}+e^{-s+t+2}}{2 e^2-2}, & -1<t<s\leq 1. \\
\end{dcases}\]
Finally, the Green's function of problem \eqref{pex1} is
\begin{align*}\ol G(t,s) & =R_\vdash G(t,s)  =\frac{\partial G}{\partial t}(-t,s)+\sqrt 2 G(-t,s)-G(t,s) \\ & =\begin{dcases}
		\frac{e^{-s-t} \left[\left(\sqrt{2}-1\right) \left(-e^{2 (s+t+1)}\right)+e^{2 s+2}+e^{2 t}-\sqrt{2}-1\right]}{2 \left(e^2-1\right)}, & |t|\le-s, \\
		\frac{e^{-s-t} \left[\left(\sqrt{2}-1\right) \left(-e^{2 (s+t)}\right)+e^{2 s+2}+e^{2 t}-\left(1+\sqrt{2}\right) e^2\right]}{2 \left(e^2-1\right)}, & |s|<t, \\
		\frac{e^{-s-t} \left[\left(\sqrt{2}-1\right) \left(-e^{2 (s+t+1)}\right)+e^{2 s}+e^{2 t+2}-\sqrt{2}-1\right]}{2 \left(e^2-1\right)}, & |s|<-t,\\
		\frac{e^{-s-t} \left[\left(\sqrt{2}-1\right) \left(-e^{2 (s+t)}\right)+e^{2 s}+e^{2 t+2}-\left(1+\sqrt{2}\right) e^2\right]}{2 \left(e^2-1\right)}, &  |t|\le s. \\
\end{dcases}
\end{align*}
Hence, the solution of problem \eqref{pex1} is $u(t):= $
\begin{align*}  -\frac{e^{-t} \left(-2 \left(1+\sqrt{2}\right) t+e^2 \left(2 \left(1+\sqrt{2}\right) t+3 \sqrt{2}\right)+e^{2 t} \left(-2 t+e^2 \left(2 t+\sqrt{2}-4\right)-\sqrt{2}\right)+\sqrt{2}+4\right)}{4 \left(e^2-1\right)}.\end{align*}
\end{exa}

\end{document}